\documentclass[a4paper]{amsart}
\usepackage{amsmath}
\usepackage{amssymb}

\usepackage{amsthm}

\theoremstyle{plain}
\newtheorem{theorem}{\bf Theorem}[section]
\newtheorem{proposition}[theorem]{\bf Proposition}
\newtheorem{lemma}[theorem]{\bf Lemma}

\theoremstyle{definition}

\newcommand{\N}{\mathbb N}

\newcommand{\DP}{\negthinspace : \negthinspace}

\renewcommand{\P}{\mathbb P}

\begin{document}


\title{The critical number of finite abelian groups}




\author{Michael Freeze and Weidong Gao and Alfred Geroldinger}

\thanks{}

%
%



\begin{abstract}
Let G be an additive, finite abelian group. The critical number
$\mathsf{cr}(G)$ of $G$ is the smallest positive integer $\ell$
such that for every subset $S \subset G \setminus \{0\}$ with $|S|
\ge \ell$ the following holds:  Every element of $G$ can be written
as a nonempty sum of distinct elements from $S$. The critical number
was first studied by P. Erd\H{o}s and H. Heilbronn in 1964, and due
to the contributions of many authors the value of $\mathsf {cr}(G)$
is known for all finite abelian groups $G$ except for $G \cong
\mathbb{Z}/pq\mathbb{Z}$ where $p,q$ are primes such that
$p+\lfloor2\sqrt{p-2}\rfloor+1<q<2p$.  We determine that $\mathsf
{cr}(G)=p+q-2$ for such groups.
\end{abstract}


\maketitle

\bigskip
\section{Introduction and Main Results} \label{1}
\bigskip

Let $G$ be  an additive, finite abelian group. The critical number
$\mathsf {cr}(G)$ of $G$ is the smallest positive integer $\ell$
such that  every subset $S \subset G \setminus \{0\}$ with $|S| \ge
\ell$ has the following property: every element of $G$ can be
written as a nonempty sum of distinct elements from $S$.

The critical number was first studied by P. Erd\H{o}s and H.
Heilbronn (see \cite{Er-He64}) for cyclic groups of prime order in
1964. After main contributions by H.B. Mann, J.E. Olson, G.T.
Diderrich, Y.F. Wou, J.A. Dias da Silva, Y. ould Hamidoune and W.
Gao, the precise value of $\mathsf {cr} (G)$ (in terms of the group
invariants of $G$) was determined, apart from cyclic groups of order
$pq$ where $p$ and $q$ are primes with $p + \lfloor 2\sqrt{p-2}
\rfloor +1 < q < 2p$. We settle this remaining case in the following
Theorem \ref{originalResult}. Its proof is based on ideas of G.T.
Diderrich (developed in his work on cyclic groups of order $pq$) and
on the solution of the Erd\H{o}s-Heilbronn Conjecture by J.A. Dias
da Silva and Y. ould Hamidoune.

\medskip
\begin{theorem}\label{originalResult}
Let $G$ be a cyclic group of order $pq$ where $p,q$ are primes with
$p + \lfloor 2\sqrt{p-2} \rfloor +1 < q < 2p.$ Then $\mathsf
{cr}(G)=p+q-2$.
\end{theorem}

\medskip
We consequently have the following determination of the value of the
critical number for all finite abelian groups. Apart from Theorem
\ref{originalResult}, it is based on the fundamental work of many
authors, and at the end of Section \ref{2} we will provide detailed
references to all contributions. Note that, by definition, $|G| \le
2$ implies that $\mathsf {cr} (G) = |G|$.

\medskip
\begin{theorem}\label{maintheorem}
Let $G$ be a finite abelian group of order $|G| \ge 3$, and let $p$
denote the smallest prime divisor of $|G|$.
\begin{enumerate}
\item If $|G|=p$, then $\mathsf {cr}(G)= \lfloor 2 \sqrt{p-2} \rfloor$.

\smallskip
\item In each of the following cases we have $\mathsf{cr} (G) =
      \frac{|G|}p + p-1${\rm \,:}
      \begin{itemize}
      \item $G$ is isomorphic to one of the following groups{\rm \,:} $C_3
            \oplus C_3$, $C_2 \oplus C_2$, $C_4$, $C_6$, $C_2 \oplus C_4$,
            $C_8$.

      \item $|G|/p$ is an odd prime with $2 < p < \frac{|G|}{p} \le p
            + \lfloor 2 \sqrt{p-2} \rfloor + 1$.
      \end{itemize}

\smallskip
\item In all other cases we have $\mathsf{cr} (G) =
      \frac{|G|}p + p-2$.
\end{enumerate}
\end{theorem}

\medskip
The work on the precise value of the critical number is complemented
by investigations on the structure of sets $S \subset G \setminus
\{0\}$ with $|S| \le \mathsf {cr} (G)$ and which have the property
that every group element can be written as a nonempty sum of
distinct elements from $S$. We refer to recent work of Y. ould
Hamidoune, A.S. Llad{\'o} and O. Serra, see \cite{G-H-L-S03} and
\cite{Ha-Ll-Se08a}.

\bigskip
\centerline{\it Throughout this article, let $G$ be an additively
written, finite abelian group.}

\bigskip
\section{Notation and tools from Additive Group Theory} \label{2}
\bigskip

Let $\mathbb N$ denote the set of positive integers, $\P \subset \N$
the set of prime numbers, and let $\mathbb N_0 = \mathbb N \cup \{ 0
\}$. For real numbers $a, b \in \mathbb R$, we set $[a, b] = \{ x
\in \mathbb Z \mid a \le x \le b\}$.  For $n \in \mathbb N$, let
$C_n$ denote a cyclic group with $n$ elements. Throughout, all
abelian groups will be written additively.

Let $A, B \subset G$ be nonempty  subsets. Then  $A+B = \{ a+b \mid
a \in A, b \in B\}$ denotes their {\it sumset}.  The set $A$ is
called an {\it arithmetic progression with difference} $d \in G$ if
there is some $a \in G$ such that $A = \{ a+ \nu d \mid \nu \in [0,
|A|-1] \}$. If $A = \{a_1,\ldots,a_{\ell}\}$ and $k\in \mathbb{N},$
we denote the restricted sumset by
\[
\Sigma_{k}(A) = \{ \sum_{i\in I}a_i \mid I \subset [1,\ell] \textrm{
with } |I|=k\} \quad \text {and write} \quad  \Sigma(A) =
\displaystyle \bigcup_{k \ge 1} \Sigma_{k}(A) \,.
\]
In particular, $A = \emptyset$ if and only if $\Sigma (A) =
\emptyset$, and for convenience we set $\Sigma_0 (A) = \{0\}$. Thus
in more technical terms, the {\it critical number} $\mathsf {cr}
(G)$ is the smallest integer $\ell \in \mathbb N$ such that every
subset $S \subset G \setminus \{0\}$ with $|S| \ge \ell$ satisfies
$\Sigma (S) = G$.

\smallskip
Now we provide the background necessary to prove Theorem
\ref{originalResult}. We start with the classical addition theorem
of Cauchy-Davenport (see \cite[Corollary 5.2.8]{Ge-HK06a}).

\medskip
\begin{theorem}\label{CDext}$($Cauchy-Davenport$)$
Let $G$ be  prime cyclic  of order $p,$ $s \in \mathbb{N}_{\ge 2},$
and $A_1,\ldots,A_{s} \subset G$ nonempty subsets. Then
\[
|A_{1} + \ldots + A_{s}|\ge \min \Bigl\{ p,
\sum_{i=1}^{s}|A_{i}|-s+1 \Bigr\} \,.
\]
\end{theorem}

\medskip
In Theorem 2.1 of \cite{Di75} and a following remark, G.T. Diderrich
improved the Cauchy-Davenport bound under extra structure
assumptions on $A_1,\ldots,A_s$.

\medskip
\begin{theorem}\label{DidVerCD}$($Diderrich$)$
Let $G$ be  prime cyclic  of order $p,$ $s\in\mathbb{N}_{\ge 2}$ and
$A_1,\ldots,A_{s} \subset G$ nonempty subsets such that all subsets,
apart from one possible exception, are arithmetic progressions with
pairwise distinct nonzero differences.  Then
\[
|A_1 + \ldots + A_s|\ \ge \min \Bigl\{p, \sum_{i=1}^{s}|A_{i}|-1
\Bigr\} \,.
\]
\end{theorem}

\medskip
The  Theorem of Dias da Silva and Hamidoune settled the
Erd\H{o}s-Heilbronn Conjecture on restricted sumsets (see
\cite{Di-Ha94} for the original paper, and also  \cite[Theorems 3.4
and 3.8]{Na96b}).

\medskip
\begin{theorem}\label{DdSHam}$($Dias da Silva-Hamidoune$)$
Let $G$ be  prime cyclic of order $p$, $S \subset G$ a subset and $k
\in [1,|S|]$.
\begin{enumerate}
\item $|\Sigma_{k}(S)| \ge \min \{p,k(|S|-k)+1 \}$.

\smallskip
\item If $|S| = \lfloor \sqrt{4p-7} \rfloor$ and $k = \lfloor |S|/2 \rfloor,$ then $\Sigma_{k}(S)=G.$
\end{enumerate}
\end{theorem}

\medskip
Clearly, the second item of \ref{DdSHam} is a special case of the
first item.  Simple calculations show that $k \in [2,|S|-1],$ then
$k(|S|-k)+1 \ge |S|$ whence $|\Sigma_{k}(S)| \ge |S|.$  We use these
observations throughout the paper.

\medskip
For the convenience of the reader we offer a proof of Theorem
\ref{maintheorem} based on Theorem \ref{originalResult} and
on the fundamental work of prior authors, which is scattered in
the literature and for which we offer precise references. Moreover,
we recall the classical example showing that for $|G| > p$ we have
\[
\mathsf {cr} (G) \ge \frac{|G|}{p} + p - 2 \,,  \quad \text{where} \
p \ \text{is the smallest prime divisor of $|G|$} \,.
\]
Let $H \subset G$ be a subgroup with $(G \DP H) = p$. Then there
exist $h_1, \ldots, h_{p-2} \in G \setminus H$ such that $h_1+H =
h_i+H$ for all $i \in [1, p-2]$. Then for $S = (H \setminus \{0\})
\cup \{h_1, \ldots, h_{p-2}\}$ we have $\Sigma (S) \subset H
\cup(h_1+H) \cup \ldots \cup \bigl( (p-2)h_1 + H \bigr)$. This shows
that $|\Sigma (S)| \le (p-1)|H| < |G|$ and thus $\mathsf {cr} (G)
\ge |S| + 1 = |G|/p + p - 2$.

\medskip
\begin{proof}[Proof of Theorem \ref{maintheorem}, based on
\ref{originalResult}] Let $|G| \ge 3$ and $p$ be the smallest prime
divisor of $|G|$.

\smallskip
CASE 1: \, $G$ is cyclic of order $p$.

Note, since $p \ge 3$, we have $\sqrt{4p-7} \notin \mathbb N$ and
thus $\lfloor \sqrt{4p-7} \rfloor = \lfloor \sqrt{4p-8} \rfloor =
\lfloor 2 \sqrt{p-2} \rfloor$. Thus Theorem \ref{DdSHam} by Dias da
Silva and Hamidoune shows that $\lfloor 2 \sqrt{p-2} \rfloor$ is an
upper bound (see \cite[Corollary 4.2]{Di-Ha94} for details), and
simple examples show that the bound is sharp (see \cite[Example
4.2]{Di-Ha94} and \cite[Theorem 7]{Gr01b}).

\smallskip
CASE 2: \, $G = C_p \oplus C_p$ with $p \ge 3$.

H.B. Mann and J.E. Olson showed that $\mathsf {cr} (G) \le 2p-1$
(with equality for $p=3$), and after that  $\mathsf {cr} (G) = 2p-2$
for all $p \ge 5$ was proved by H.B. Mann and Ying Fou Wou (see
\cite{Ma-Wo86}).

\smallskip
CASE 3: \, $G = C_p \oplus C_q$ for a prime $q$ with $3 \le p < q$.

The case $q \le p + \lfloor 2 \sqrt{p-2} \rfloor + 1$ was settled by
J.R. Griggs (see \cite[Theorem 4]{Gr01b}).

The case $p + \lfloor 2\sqrt{p-2} \rfloor +1 < q < 2p$ follows from
the present Theorem \ref{originalResult}.

The case $q \ge 2p+1$ was settled by G.T. Diderrich (see
\cite[Theorem 1.0]{Di75}).

\smallskip
CASE 4: \, $|G|$ is even.

This case was settled by G.T. Diderrich and H.B. Mann in
\cite{Di-Ma73}, see also  \cite[Theorem 5]{Gr01b} for a
self-contained, simplified proof.

\smallskip
CASE 5: \, $|G|$ is odd and $|G|/p$ is composite.

Then $\mathsf {cr} (G) = |G|/p + p -2$ by W. Gao and Y. ould
Hamidoune (see \cite{Ga-Ha99}).
\end{proof}

\bigskip
\section{The setting and the strategy of the proof} \label{3}
\bigskip

First,  we fix our notations which remain valid throughout the rest
of the paper, and then we outline the strategy of the proof of
Theorem \ref{originalResult}.

\smallskip
Let $G$ be  cyclic  of order $pq$ where $p,q$ are primes with
$p+\lfloor2\sqrt{p-2}\rfloor+1<q<2p$ (which implies that $p \ge 7$)
and let $S \subset G \setminus \{0\}$ be a subset with $|S|=p+q-2$.

Let $H, K  \subset G$ be the subgroups with $(G \DP H)=p$ and $(G
\DP K)=q$. Let $s = |\{a+H \in G/H \mid a \in S \setminus H\}|$ and
pick $a_1,\ldots,a_s \in S \setminus H$ such that $|\{a_i+H \mid i
\in [1,s]\}| = s$. We set $S_0 = H \cap S$ and $S_i = (a_i+H) \cap
S$ for all $i \in [1,s]$.

Suppose that $a_1, \ldots, a_s$ and $t,r,n \in \mathbb N_0$ are
chosen in such a way that
\begin{itemize}
\item $\left|S_{1}\right| \geq \ldots \geq \left|S_{t}\right|\geq 3$,

\item $\left|S_{t+1}\right| = \ldots =\left|S_{t+r}\right|=1$ and

\item $\left|S_{t+r+1}\right| = \ldots = \left|S_{t+r+u}\right|=2$.
\end{itemize}
Notice that
\[
s=t+r+u \le p-1 \quad \text{ and} \quad
|S_0|+\sum_{i=1}^{t}|S_i|+r+2u=p+q-2=|S| \,.
\]
For an element $x \in G$ we  consider a representation
\[
x+H = \sum_{i=1}^s f_i (a_i+H) \tag{$*$}
\]
with $f_i \in [0, |S_i|]$ for all $i \in [1,s]$ and $f_1+ \ldots +
f_s > 0$. If $f_i \in \{0, |S_i| \}$, then $f_i$ is called a {\it
collapsed coefficient} and
\[
C(*) = \sum _{i=1, f_{i}\in \left\{0,
\left|S_{i}\right|\right\}}^{s}\left(\left|S_{i}\right|-1\right)
\]
is called the {\it collapse} of the representation $(*)$. We say that
{\it $G/H$ has a representation with collapse $C \in \mathbb N_0$} if
every $x \in G$ has a representation $(*)$ and $C$ is the maximum of
the collapses $C(*)$.

\medskip
The strategy of the proof is as follows. First we settle the very
simple case where $|S_0| \ge \lfloor2\sqrt{q-2}\rfloor$. After
supposing that $|S_0| \le \lfloor2\sqrt{q-2}\rfloor -1 $ we follow
the ideas of G.T. Diderrich and proceed in two steps:

\medskip
\begin{itemize}
\item[{\bf 1.}] First, we show that $G/H$ has a representation with some
      collapse $C \in \N_0$ (see Lemmas \ref{KnMiddleCollapse1},
\ref{largetsmallcollapse}, \ref{KnSmallCollapse1}).

\smallskip
\item[{\bf 2.}] For $x \in G$ and a representation $(*)$  we show that
\[
\left| \left( \Sigma (S_0) \cup \{0\} \right)
+\Sigma_{f_{1}}(S_{1})+...+\Sigma_{f_{s}}(S_{s})\right| \ge q \,.
\]
\end{itemize}

\medskip
\noindent
Suppose that {\bf 1.} and {\bf 2.} are settled. Notice
that
\[
\begin{aligned}
\left(  \Sigma (S_0) \cup \{0\} \right)
+\Sigma_{f_{1}}(S_{1})+...+\Sigma_{f_{s}}(S_{s}) & \subset H +
f_{1}(a_{1}+H)+...+f_{s}(a_{s}+H)
\\
 & =H+\sum_{i=1}^{s}f_{i}(a_{i}+H)=x+H \,.
\end{aligned}
\]
Thus {\bf 2.} implies that we have equality in the above inclusion.
Therefore
\[
x+H = \left( \Sigma (S_0) \cup \{0\}
 \right)+\Sigma_{f_{1}}(S_{1})+...+\Sigma_{f_{s}}(S_{s})\subset \Sigma(S)
\,,
\]
and together with {\bf 1.} we obtain $G = \Sigma (S)$.

\bigskip
\section{Proof of Theorem \ref{originalResult}} \label{4}
\bigskip

We start with a simple special case.

\medskip
\begin{proposition}\label{KnBig}
If $|S_0| \ge \lfloor2\sqrt{q-2}\rfloor$, then $\Sigma (S) = G$.
\end{proposition}

\begin{proof}
Suppose that $|S_0| \ge\lfloor2\sqrt{q-2}\rfloor$.  Since, by
Theorem \ref{maintheorem}.1, $\mathsf
{cr}(H)=\lfloor2\sqrt{q-2}\rfloor$, it follows that $\Sigma (S_0) =
H$. Since $|S \setminus H|\ge p+q-2-(q-1)=p-1$, we can choose $p-1$
distinct elements $b_1, \cdots, b_{p-1} \in S\setminus H$. For $i
\in [1, p-1]$ we set $W_i = \{ 0+H, b_i+H\} \subset G/H$, and by
Theorem \ref{CDext} we obtain that
\[
|\Sigma (W_1+ \ldots + W_{p-1})| \ge \min \{p, 2(p-1)-(p-1)+1\} = p
\,.
\]
Thus it follows that
\[
\Sigma (S) \supset  \Sigma (S_0) + \bigl( \Sigma (\{b_1, \ldots,
b_{p-1}\}) \cup \{0\} \bigr) = G \,. \qedhere
\]
\end{proof}

\medskip
Hence from now on we may assume that  $|S_0| \le \mathsf {cr}(H) -1
$, and we proceed in the two steps described above.

\medskip
\begin{lemma}\label{largetsmallcollapse}
If  $t\ge\lfloor2\sqrt{p-2}\rfloor$, then $G/H$ has  a
representation  with collapse $C=0$.
\end{lemma}

\begin{proof}
By Theorem \ref{maintheorem}.1, we have $t \ge \lfloor 2\sqrt{p-2}
\rfloor = \mathsf {cr}(G/H)$ and thus $\Sigma ( \{a_1+H, \ldots,
a_t+H\}) = G/H$. Pick some $x \in G$. Then there exists a nonempty
subset $I \subset [1,t]$ such that
\[
x-(a_1+ \ldots + a_s) + H = \sum_{i \in I} (a_i+H)
\]
and hence
\[
x+H = \sum_{i \in I} 2 (a_i+H) + \sum_{i \in [1,t] \setminus I}
(a_i+H) + \sum_{i=t+1}^{s} (a_i+H) \,. \tag{$**$}
\]
Since $|S_i| \ge 3$ for all $i \in [1,t]$, the representation $(**)$
has collapse $C (**) = 0$.
\end{proof}

\medskip
\begin{lemma}\label{KnMiddleCollapse1}
If $|S_0| \le\lfloor 2\sqrt{q-2} \rfloor - 1$,  then  $G/H$ has a
representation with collapse $C\le 1$.
\end{lemma}

\begin{proof}
 Suppose that
$|S_0| \le\lfloor2\sqrt{q-2}\rfloor-1$. We construct sets  $A_1,
\ldots, A_{t+r}$ and $D$ as follows:

\begin{eqnarray*}
A_{i}&=&\{a_{i}+H,\ldots,(|S_i|-1)a_{i}+H \} \subset G/H \ \ \textrm{for}\  i\in [1, t];\\
A_{i}&=&\{H, a_{i}+H \} \subset G/H \ \ \textrm{for}\ \ i \in [t+1, t+r] ;\\
D&=&\{b_{0},b_{0}-b_{1},b_{0}-b_{2},\ldots,b_{0}-b_{u}\} \subset G/H
\end{eqnarray*}
where $b_{j}=a_{t+r+j}+H$ for $j \in [1, u]$, and $\displaystyle
b_{0}=\sum_{j=1}^{u}b_{j} + H$. We   assert that $D +
\sum_{i=1}^{t+r} A_{i} = G/H$. Clearly, this implies that $G/H$ has
a representation with collapse $C \le 1$.

\smallskip
Assume to the contrary, that ${D}+\sum_{i=1}^{t+r} {A_{i}}
\subsetneq G/H$. Applying the Cauchy-Davenport Theorem and Theorem
\ref{DidVerCD}, we have
\begin{eqnarray*}
|{D}+\sum_{i=1}^{t+r} {A_{i}}|&\ge&| {D}|+|\sum_{i=1}^{t+r} {A_i}|-1\\
&\ge&|D|+\sum_{i=1}^{t+r}|{A_{i}}|-2\\
&=&u+\sum_{i=1}^{t}|S_i|-t+2r-1 \,,
\end{eqnarray*}
and hence
\[
u+\sum_{i=1}^{t}|S_i|-t+2r-1\le p-1 \,. \tag{$***$}
\]
Since by our constructions,
\[
p+q-2= |S_0| + \sum_{i=1}^{t} |S_i|+r+2u \,,
\]
we can solve this equation for $\sum_{i=1}^{t}|S_i|$, yielding
\begin{eqnarray*}
u+(p+q-2-|S_0|-r-2u)-t+2r-1&\le&p-1\\
q-u-t+r&\le&|S_0|+2\\
q-(u+t+r)+2r&\le& |S_0|+2.\\
\end{eqnarray*}
Therefore we have
\[
q- (u+t+r) + (2r-1) \le |S_0|+1 \,.
\]
We distinguish two cases.

\smallskip
\noindent CASE 1: \, $|S_0|\le \lfloor2\sqrt{q-2}\rfloor-2$ or
 $s \le p-2$ or $r \ge 1$.

Using $u+t+r=s\le p-1$ and the assumption of CASE 1  we obtain
\[
q-p+1\le\lfloor2\sqrt{q-2}\rfloor \,.
\]
Here, since $q-p+1$ is
positive, squaring both sides preserves the inequality, giving us
\begin{eqnarray*}
(q-p)^{2}+2(q-p)+1&\le&4(q-2)\\
q^{2}-2pq+p^{2}+2q-2p+1&\le&4q-8\\
q^{2}-2pq-2q+p^{2}-2p+9&\le&0\\
q^{2}-(2p+2)q+(p^{2}-2p+9)&\le&0.
\end{eqnarray*}
By considering this as a quadratic in terms of $q$, we can apply the
quadratic formula to find that
\[
q\le p+1+2\sqrt{p-2} \,.
\]
Since $q$ and $p+1$ are integers, we then have
\[
q\le p+1+\lfloor2\sqrt{p-2}\rfloor \,,
\]
which is a contradiction to the original restrictions of
$p+\lfloor2\sqrt{p-2}\rfloor+1<q<2p$.

\medskip
\noindent CASE 2: \,  $|S_0| = \lfloor2\sqrt{q-2}\rfloor-1$, $s=p-1$
and $r=0$.

Using $p-1=s=t+r+u=t+u$ and $(***)$ we obtain that
\[
p-1+(t-1) = u+2t-1 \le u  + \sum_{i=1}^t |S_i| - t - 1 \le p-1 \,,
\]
whence $t \le 1$.

If $t = 0$, then $p = u+1 = |D|$ and hence $D = G/H$, a
contradiction.

If  $t=1$, then  $u=p-2$ and (looking back at the beginning of the
proof) we get
\[
p-1 \ge |{D}+\sum_{i=1}^{t+r} {A_{i}}| = |D + A_1| \ge |D| + |A_1| -
1 = (u+1) + (|S_1|-1) - 1 = u + |S_1| - 1 \,,
\]
and hence $|S_1| \le 2$, a contradiction.
\end{proof}

\smallskip
We require the following technical Lemma.

\medskip
\begin{lemma}\label{newLemma}
Suppose that $G/H$ has a representation with collapse $C \in \mathbb
N_0$. If $(p+q-2)+  \max\{1,|S_0|-1\}  -C-s\ge q$, then for every $x
\in G$ with $C (*) \le C$ we have $|\left( \Sigma (S_0) \cup \{0\}
\right) + \sum_{i=1}^{s} \Sigma_{f_i} (S_i)| \ge q$.
\end{lemma}

\begin{proof}
For any subset $A \subset G$ we set $\bar{ A} = \{a+K \mid a \in A\}
\subset G/K$ where $K\subset G$ is the subgroup with $(G \DP K)=q$.
Clearly we have $\overline{\Sigma_k (A)} = \Sigma_k (\overline A)$
for all $k \in \mathbb N_0$, $|A| \ge |\overline A|$ and if $x \in
G$ and $A \subset x+H$, then $|A| = |\overline A|$. If $|\Sigma (
\overline{S_0}) \cup \{\overline 0\}| \ge q$, then the statement of
the Lemma follows. Suppose that $|\Sigma ( \overline{S_0} \cup
\{\overline 0\})| < q$.

We assert that $|  \Sigma (S_0) \cup \{0\} |\ge |S_0|+
\max\{1,|S_0|-1\}$. If $|S_0| \le 1$, then this is clear. Suppose
that $\overline{S_0} = \{z_1+K, \ldots, z_{\lambda}+K\}$ with $z_1,
\ldots, z_{\lambda} \in G$ and $\lambda \ge 2$. Then $\Sigma (
\overline{S_0}) \cup \{\overline 0\} = \{\overline 0, z_1+K\} +
\ldots + \{\overline 0, z_{\lambda} + K \}$, and Theorem
\ref{DidVerCD} implies that
\[
|  \Sigma (S_0) \cup \{0\} | \ge |\Sigma ( \overline{S_0}) \cup
\{\overline 0\}| \ge \min \{q, 2 \lambda - 1 \} = 2 |S_0| - 1 \,.
\]
Let $x \in G$ with representation $(*)$ and let $i \in [1,s]$. If
$f_i$ is a collapsed coefficient, then  $|\Sigma_{f_i} (S_i)|=
|\Sigma_{f_i} (\overline{S_i})|= 1$. If $f_{i}$ is not a collapsed
coefficient, then the observation after Theorem \ref{DdSHam} gives
us
\[
|\Sigma_{f_i} (S_i)| \ge |\Sigma_{f_i} (\overline{S_i)}| \ge
|\overline{S_{i}}| = |S_{i}| \,.
\]
Thus we obtain
\[
| \Sigma (\overline{S_0}) \cup \{\overline 0\}| +
 \sum_{i=1}^{s}| \Sigma_{f_i} (\overline{S_i})| \ge \sum_{i=0}^{s} |S_i| + \max\{1,|S_0|-1\}
 -C \,,
\]
and by Theorem \ref{CDext} we have
\[
\begin{aligned}
|\left( \Sigma (S_0) \cup \{0\} \right) + \sum_{i=1}^{s}
\Sigma_{f_i} (S_i)| & \ge |\left( \Sigma (\overline{S_0}) \cup
\{\overline 0\} \right) +
\sum_{i=1}^{s} \Sigma_{f_i} (\overline{S_i})| \\
& \ge  \min \{q, |\Sigma (\overline{S_0})
\cup \{\overline 0\} | + \sum_{i=1}^{s}|\Sigma_{f_i} (\overline{S_i})|-s \} \\
& \ge  \min \{q, \sum_{i=0}^{s}|S_i|+ \max\{1,|S_0|-1\} -C-s \} \\
& =  \min \{q, p+q-2+ \max\{1,|S_0|-1\}  -C-s \} \,.
\end{aligned}
\]
So if $p+q-2+ \max\{1,|S_0|-1\}  -C-s\ge q$, then the assertion
follows.
\end{proof}

\medskip
\begin{proposition}\label{KnMiddle}
If $3\le |S_0| \le \lfloor2\sqrt{q-2}\rfloor - 1$, then $\Sigma (S)
= G$.
\end{proposition}

\begin{proof}
Since $|S_0| \le  \lfloor2\sqrt{q-2}\rfloor - 1$, Lemma
\ref{KnMiddleCollapse1}  gives us a representation of $G/H$ with
collapse $C\le1$. Notice that since $|S_0| \ge 3$, we have $p+q-2+
\max\{1,|S_0|-1\} -C-s\ge p+q-2+2-1-(p-1)=q$. Thus the assertion
follows from Lemma \ref{newLemma}.
\end{proof}

\smallskip
Now consider the case $|S_0| \le 2$, we contemplate two subcases.
First take the case where $|S_1| \le3$.

\medskip
\begin{proposition}\label{KnSmallAllSm}
If $|S_0| \le 2$ and $|S_1| \le3$, then $\Sigma (S) = G$.
\end{proposition}

\begin{proof}
Since $q \ge 5$, we have  $|S_0| \le2<\lfloor2\sqrt{q-2}\rfloor$.
Thus Lemma \ref{KnMiddleCollapse1} implies that  there is a
representation of $G/H$ with collapse $C\le1$. Thus it remains to
verify the assumption of Lemma \ref{newLemma}, and thus we have to
show that
\[
(p+q-2)+ \max\{1,|S_0|-1\} -C-s\ge q \,.
\]
Note that $|S_0| \le2$ implies that $\max\{1,|S_0|-1\} = 1$. We have
$(p+q-2)+ \max\{1,|S_0|-1\} -C-s\ge q$ for $s \le p-2$. Consider the
case $s=p-1$. Since
\[
\begin{aligned}
p+q-2&=|S_0|+\sum_{i=1}^{t}|S_i|+r+2u\\
&\le2+3t+r+2u\\
&=2+s+2t+u\\
&=p+1+2t+u \,,
\end{aligned}
\]
it follows that
\[
\begin{aligned}
q-2 &\le  1 + 2t + u \\
&= 1 + t + (t+u+r) -r\\
&= 1+t+(p-1)-r \\
&= t+p-r.
\end{aligned}
\]
Since $q \ge p + \lfloor 2 \sqrt{p-2} \rfloor +2$, we see that
\begin{eqnarray*}
t &\ge& q-p-2 \\
&\ge& \lfloor 2\sqrt{p-2} \rfloor.
\end{eqnarray*}

Consequently, Lemma \ref{largetsmallcollapse} implies that we have
collapse $C=0$. Putting all together we obtain
\[
(p+q-2)+ \max\{1,|S_0|-1\}  -C-s\ge (p+q-2)+1-0-(p-1)=q \,,
\]
and hence the assumption of Lemma \ref{newLemma} is satisfied.
\end{proof}

\medskip
Finally, we address the remaining case where $|S_0| \le2$ and $|S_1|
\ge 4$.

\medskip
\begin{lemma}\label{KnSmallCollapse1}
If $|S_0| \le 2$ and $|S_1| \ge 4$, then for every $x \in G$ there
is a representation $(*)$ of $x+H$ with  $f_1 \in [2, |S_1| -2]$.
\end{lemma}

\begin{proof}
We argue as in Lemma \ref{KnMiddleCollapse1}.
 Suppose that
$|S_0| \le 2$ and $|S_1| \ge 4$. We construct sets  $A_1, \ldots,
A_{t+r}$ and $D$ as follows:

\begin{eqnarray*}
A_{1}&=&\{2a_{1}+H,\ldots,(|S_1|-2)a_{1}+H\} \subset G/H \,, \\
A_{i}&=&\{a_{i}+H,\ldots,(|S_i|-1)a_{i}+H \} \subset G/H \ \ \textrm{for}\  i\in [2, t];\\
A_{i}&=&\{H, a_{i}+H \} \subset G/H \ \ \textrm{for}\ \ i \in [t+1, t+r] ;\\
D&=&\{b_{0},b_{0}-b_{1},b_{0}-b_{2},\ldots,b_{0}-b_{u}\} \subset G/H
\end{eqnarray*}
where $b_{j}=a_{t+r+j}+H$ for $j \in [1, u]$, and $\displaystyle
b_{0}=\sum_{j=1}^{u}b_{j} + H$. It suffices to show that $D +
\sum_{i=1}^{t+r} A_{i} = G/H$. Applying the Cauchy-Davenport Theorem
and Theorem \ref{DidVerCD}, we have
\begin{eqnarray*}
|{D}+\sum_{i=1}^{t+r} {A_{i}}|&\ge& \min \{p, | {D}|+|\sum_{i=1}^{t+r} {A_i}|-1 \} \\
&\ge& \min \{p, |D|+\sum_{i=1}^{t+r}|{A_{i}}|-2 \} \\
&=& \min \{p, u+\sum_{i=1}^{t}|S_i|-t+2r-3 \} \,.
\end{eqnarray*}
Recall that $p+q-2=|S_0|+\sum_{i=1}^{t}|S_i|+r+2u$ implies $
\sum_{i=1}^{t}|S_i|=p+q-2-|S_0|-r-2u$. Now we have
\begin{eqnarray*}
u+\sum_{i=1}^{t}|S_{i}|-t+2r-3&=&u+(p+q-2-|S_0|-r-2u)-t+2r-3 \\
&=&p+q-5-|S_0|+r-u-t\\
&=&p+q-5-|S_0|+r-(s-r)\\
&=&p+q-5-|S_0|+2r-s.\\
\end{eqnarray*}
Since $s\le p-1$, we see that
\begin{eqnarray*}
p+q-5-|S_0|+2r-s&\ge&p+q-5-|S_0|+2r-p+1\\
&=&q-4-|S_0|+2r\\
&\ge&p+\lfloor2\sqrt{p-2}\rfloor-2-|S_0|+2r
\end{eqnarray*}
for the given values of primes $p,q$.  This gives us
\begin{eqnarray*}
u+\sum_{i=1}^{t}|S_{i}|-t+2r-3& \ge &p+\lfloor2\sqrt{p-2}\rfloor-2-|S_0|+2r\\
&\ge&p+\lfloor2\sqrt{p-2}\rfloor-4+2r\\
&\ge&p+\lfloor2\sqrt{p-2}\rfloor-4\\
&\ge&p \,.
\end{eqnarray*}
\end{proof}

\medskip
\begin{proposition}\label{KnSmallSomeBig}
If $|S_0| \le 2$ and $|S_1| \ge 4$, then $\Sigma (S) = G$.
\end{proposition}

\begin{proof}
By Lemma \ref{KnSmallCollapse1} it remains to show that
\[
|\left( \Sigma (S_0) \cup \{0\} \right) + \sum_{i=1}^{s}
\Sigma_{f_i} (S)| \ge q \,.
\]
As in the proof of Lemma \ref{newLemma}, we set, for any subset $A
\subset G$, $\bar{ A} = \{a+K \mid a \in A\} \subset G/K$ , and we
use all observations made before.  Theorem \ref{CDext} implies that
\[
|\left( \Sigma (\overline{S_0}) \cup \{\overline 0\} \right) +
\sum_{i=1}^{s} \Sigma_{f_i} (\overline{S_i})|
     \ge \min \{q, |\Sigma (\overline{S_0}) \cup \{\overline 0\}|+|\Sigma_{f_1}
(\overline{S_1})|+\sum_{i=2}^{s}|\Sigma_{f_i} (\overline{S_i})|-s \}
\,.
\]
By Lemma \ref{KnSmallCollapse1}  we have $f_1 \in
 [2, |S_1|-2]$, and by Theorem \ref{DdSHam} we get
\[
|\Sigma_{f_1}(\bar{{S_{1}}})| \ge \min\{q,\ f_1 |S_1|- f_1^{2} + 1\}
\,.
\]
To find a lower bound on this inequality, we consider the minimum
value of the quadratic expression $f_1 |S_1| - f_1^{2} + 1$ over the
interval $[2, |S_1|-2]$.  Since the leading term is negative, the
minimum value will occur when $f_1=2$ or $f_1=|S_1|-2$.  Hence $f_1
|S_1|- f_1^{2} + 1 \ge  2 |S_1| - 3 \ge |S_1|+1$ because $|S_1| \ge
4$. Now we have
\[
\begin{aligned}
|\left( \Sigma (S_0) \cup \{0\} \right) + \sum_{i=1}^{s}
\Sigma_{f_i} (S)| &\ge |\left( \Sigma (\overline{S_0}) \cup
\{\overline 0\} \right) +
\sum_{i=1}^{s} \Sigma_{f_i} (\overline{S_i})| \\
&\ge \min \{q, |\Sigma (\overline{S_0}) \cup \{\overline
0\}|+|\Sigma_{f_1}
(\overline{S_1})|+\sum_{i=2}^{s}|\Sigma_{f_i} (\overline{S_i})|-s \} \\
&\ge \min \{q, |\Sigma (S_0) \cup
\{0\}|+(|S_1|+1)+\sum_{i=2}^{s}|S_i|-1-s \} \,,
\end{aligned}
\]
where we subtract one for a possible collapsed coefficient yielding
$\Sigma_{f_i} (S_i) =\{0\}$ for some $i \in [2, s]$.  Therefore we
obtain that
\[
\begin{aligned}
|\left( \Sigma (S_0) \cup \{0\} \right) + \sum_{i=1}^{s}
\Sigma_{f_i} (S)| & \ge \min \{q, (|S_0|+ \max\{1,|S_0|-1\}  +(|S_1|+1)+\sum_{i=2}^{s}|S_i|-1-s \} \\
& = \min \{q, p+q-2+ \max\{1,|S_0|-1\} -s \} \\
& \ge  \min \{q,  p+q-2+1-(p-1) \} \\
& =  q \,. 
\end{aligned}
\]
\end{proof}

\medskip
Now the proof of Theorem \ref{originalResult} follows by a simple
combination of the previous propositions.

\medskip
\begin{proof}[Proof of Theorem \ref{originalResult}]
Let $G$ be  cyclic  of order $pq$ where $p,q$ are primes with
$p+\lfloor2\sqrt{p-2}\rfloor+1<q<2p,$ and let $S \subset G \setminus
\{0\}$ be a subset with $|S|=p+q-2$. We use all notations as
introduced at the beginning of Section \ref{3}.

If $|S_0| \ge \lfloor 2\sqrt{q-2} \rfloor$, then Proposition
\ref{KnBig} implies that $\Sigma (S) = G$.

If $3 \le |S_0| \le \lfloor 2\sqrt{q-2} \rfloor - 1$, then
Proposition \ref{KnMiddle} yields that $\Sigma (S) = G$.

Consider now the case $|S_0| \le 2$.  If additionally we have $|S_1|
\le 3$, then Proposition \ref{KnSmallAllSm} yields that $\Sigma (S)
= G$. On the other hand, if $|S_1| \ge 4$, then Proposition
\ref{KnSmallSomeBig} yields that $\Sigma (S) = G$.
\end{proof}

\bigskip
{\bf Acknowledgement} This work was supported by the Austrian
Science Fund FWF, Project No. P18779-N13.



\end{document}